\documentclass[11pt,a4paper]{article}
\usepackage[utf8]{inputenc}
\usepackage[T1]{fontenc}
\usepackage[english,french]{babel}
\addto\captionsfrench{%
}
\usepackage{xcolor}
\usepackage{amsmath}
\usepackage{amssymb}
\usepackage{amsfonts}
\usepackage{amsthm}
\usepackage{mathtools}
\usepackage{mathrsfs}
\usepackage{bm}
\usepackage{tensor}
\usepackage[a4paper,margin=2.7cm]{geometry}
\usepackage{microtype}
\usepackage{setspace}
\usepackage[
    colorlinks=true,
    linkcolor=blue!60!black,
    citecolor=blue!60!black,
    urlcolor=blue!60!black
]{hyperref}

\usepackage[nameinlink,noabbrev]{cleveref}
\usepackage{booktabs}
\usepackage{enumitem}
\usepackage{authblk}

\newtheorem{theorem}{Theorem}[section]

\newtheorem{corollary}[theorem]{Corollary}
\theoremstyle{definition}
\newtheorem{definition}[theorem]{Definition}
\newtheorem{remark}[theorem]{Remark}

\newcommand{\R}{\mathbb{R}}
\newcommand{\Lie}{\mathcal{L}}

\newcommand{\Ric}{\operatorname{Ric}}
\newcommand{\Scal}{\operatorname{Scal}}

\newcommand{\Hess}{\operatorname{Hess}}
\newcommand{\doi}[1]{%
    \href{https://doi.org/#1}{\texttt{#1}}%
}

\title{\textbf{Characterization of contact structures for the dual form of the vector field of an $h$-Ricci-Bourguignon soliton on the manifold $\mathbb{D}^2 \times \mathbb{R}$}}

\author[1]{Mafal Ndiaye Diop\thanks{E-mail: \texttt{mafaldiop95@gmail.com}}}
\author[1]{Abdou Bousso\thanks{E-mail: \texttt{abdoukskbousso@gmail.com}}}
\author[2]{Ameth Ndiaye\thanks{E-mail: \texttt{ameth1.ndiaye@ucad.edu.sn}}}

\affil[1]{Département de Mathématiques et Informatique, Université Cheikh Anta Diop de Dakar, Sénégal}
\affil[2]{Département de Mathématiques, FASTEF, Université Cheikh Anta Diop de Dakar, Sénégal}

\date{}

\begin{document}

\maketitle
\vspace{-0.5cm}
\renewcommand{\abstractname}{Abstract}
\begin{abstract}
We study Ricci-Bourguignon $h$-solitons on the product manifold $\mathbb{D}^2\times\mathbb{R}$, where $\mathbb{D}^2$ denotes the Poincaré disk equipped with its standard hyperbolic metric. We first show that this manifold admits a gradient Ricci-Bourguignon $h$-soliton only if the non-zero function $h$ depends solely on the manifold's last coordinate. Next, in the case where $h:=1$, we explicitly determine the vector field that makes $\mathbb{D}^2\times\mathbb{R}$ a Ricci-Bourguignon soliton. Finally, we establish the necessary and sufficient condition for the dual 1-form of this vector field to define a contact structure. This condition corresponds to the existence of a function which we explicitly determine that is independent of the manifold's second local coordinate, thereby allowing us to explicitly calculate its Reeb vector field $\xi$.
\end{abstract}

\vspace{0.3cm}


\vspace{0.2cm}


\section{Introduction}

The study of geometric flows and their self-similar solutions—initiated by Hamilton \cite{Hamilton} with the Ricci flow and pursued notably by Chow \cite{Chow1992} in the context of the Yamabe flow—occupies a central place in modern differential geometry. For the fundamentals of Riemannian geometry, connections, and curvature, we refer to the classic texts by Besse \cite{Besse}, do Carmo \cite{doCarmo}, and Lee \cite{Lee}. An $h$-Ricci-Bourguignon soliton, denoted by $(M,g,X,\lambda,\rho)$, is characterized by the equation
\begin{equation}
h\mathcal{L}_X g + 2\operatorname{Ric} = 2(\lambda + \rho \Scal) g,
\label{SRB}
\end{equation}
where $X$ is a vector field, $\operatorname{Ric}$ denotes the Ricci tensor, $\Scal$ is the scalar curvature, and $\lambda, \rho \in \mathbb{R}$ are real constants. When the vector field $X$ derives from a potential $f$ of class at least $\mathcal{C}^2$ (i.e., $X = \nabla f$), the soliton is called a gradient soliton, and the equation is then expressed in terms of the Hessian of $f$:
\begin{equation}
h\operatorname{Hess}(f) + \operatorname{Ric} = (\lambda + \rho \Scal) g.
\label{SRBG}
\end{equation}

In recent years, particular attention has been paid to the analysis of these geometric structures on various classes of spaces, drawing upon the topology of 3-manifolds and the geometry of product manifolds (see, for example, Thurston \cite{Thurston1997} and O'Neill \cite{ONeill}). In this context, the study of solitons on hyperbolic spaces and product manifolds such as the Poincaré half-plane \cite{BoussoNdiaye2026Hn}, general hyperbolic spaces \cite{BoussoNdiayeJDSGT2025}, the Lie group $\mathbb{H}^2 \times \mathbb{R}$ \cite{BoussoNdiayeH2R2025}, the group $\mathrm{Sol}_3$ \cite{BoussoNdiayeSol3}, or the Poincaré disk via $(h,\eta)$-extensions \cite{DiopBoussoNdiayeMandal2026} has highlighted the decisive influence of the underlying geometry on the admissibility and rigidity of the solutions.

Continuing this line of inquiry, we consider the product manifold $M = \mathbb{D}^2 \times \mathbb{R}$, where $\mathbb{D}^2$ denotes the Poincaré disk equipped with its standard hyperbolic metric. We first analyze the existence of gradient-type $h$-Ricci-Bourguignon solitons, demonstrating that this manifold admits such solitons only if the non-zero function $h$ depends solely on the manifold's last coordinate. Next, focusing on the canonical case where $h: = 1$, we explicitly determine the vector field that makes $\mathbb{D}^2 \times \mathbb{R}$ a Ricci-Bourguignon soliton. Finally, we explore the geometric properties of the dual 1-form associated with this vector field, establishing the necessary and sufficient condition for it to define a contact structure. This condition implies the existence of an explicitly determined function independent of the second local coordinate, enabling us to calculate its Reeb field $\xi$ exactly.
\section{Preliminaries}

Consider the Poincaré disk $\mathbb{D}^2
=
\left\{
(x,y)\in\mathbb{R}^2:
x^2+y^2<1
\right\}.$

In polar coordinates $\begin{cases}  x=r\cos\theta,\\
y=r\sin\theta,
\end{cases}$ where $r \in [0, 1)$ and $\theta \in [0, 2\pi)$, the standard hyperbolic metric of the disk is $ds^2=\frac{4}{(1-r^2)^2}\left(dr^2+r^2d\theta^2
\right).$

Thus, the product metric on $\mathbb{D}^2\times\mathbb{R}$ is: 
\begin{equation}
g=\frac{4}{(1-r^2)^2}
\left(
dr^2+r^2d\theta^2
\right)
+dt^2.
\label{eq:metric}
\end{equation}
We use the local coordinates $(x_1,x_2,x_3)=(r,\theta,t).$

Since the matrix $(g_{ij})$ is diagonal, its inverse is $(g^{ij})=\begin{pmatrix}
\frac{(1-r^2)^2}{4}&0&0\\
0&\frac{(1-r^2)^2}{4r^2}&0\\
0&0&1
\end{pmatrix}.$

Applying the formula for Christoffel symbols given by
\begin{equation}
\Gamma_{ij}^{k}
=
\frac12
\sum_{\ell=1}^{3}
g^{k\ell}
\left(
\frac{\partial g_{\ell i}}{\partial x^j}
+
\frac{\partial g_{\ell j}}{\partial x^i}
-
\frac{\partial g_{ij}}{\partial x^\ell}
\right).
\label{eq:christoffel}
\end{equation}

The only non-zero Christoffel symbols are therefore $$\Gamma_{11}^{1}=\frac{2r}{1-r^2},
\quad\Gamma_{22}^{1}=-\frac{r(1+r^2)}{1-r^2},
\quad\Gamma_{12}^{2}=\Gamma_{21}^{2}=\frac{1+r^2}{r(1-r^2)}.$$

Since the metric is a product $(g=ds^2+dt^2)$ and the $\mathbb{R}$ direction is flat, we have
\begin{equation}
(\Ric_{ij})
=
\begin{pmatrix}
-\dfrac{4}{(1-r^2)^2} & 0 & 0\\[2mm]
0 & -\dfrac{4r^2}{(1-r^2)^2} & 0\\[2mm]
0&0&0
\end{pmatrix}.
\label{eq:Ricci-matrix}
\end{equation}
This is because the Gaussian curvature of the Poincaré disk is equal to $-1$.

The scalar curvature is therefore equal to:
{\small$$\Scal=g^{11}\Ric_{11}+g^{22}\Ric_{22}+g^{33}\Ric_{33}=\frac{(1-r^2)^2}{4}\left(
-\frac{4}{(1-r^2)^2}\right)+\frac{(1-r^2)^2}{4r^2}\left(-\frac{4r^2}{(1-r^2)^2}\right)=-1-1=-2.$$}
Let us now consider a function $f:M\to\mathbb{R}$ on the product manifold $M = \mathbb{D}^2 \times \mathbb{R}$ that is at least of class $\mathcal{C}^2$. The components of the Hessian of $f$, denoted $\Hess(f)$, are defined by the formula in local coordinates:
\begin{equation}
(\Hess(f))_{ij}
=\nabla_i \partial_j f=\frac{\partial^2 f}{\partial x^i \partial x^j}-\Gamma_{ij}^{k} \frac{\partial f}{\partial x^k}.
\label{eq:hessian-def}
\end{equation}

Using the non-zero Christoffel symbols determined previously, let us write out each component $(\Hess(f))_{ij}$ : $$\begin{cases}
    (\Hess(f))_{11}
    =
    \frac{\partial^2 f}{\partial r^2}
    -
    \Gamma_{11}^{1} \frac{\partial f}{\partial r}
    =
    \frac{\partial^2 f}{\partial r^2}
    -
    \frac{2r}{1-r^2} \frac{\partial f}{\partial r},\\
     (\Hess(f))_{22}
    =
    \frac{\partial^2 f}{\partial \theta^2}
    -
    \Gamma_{22}^{1} \frac{\partial f}{\partial r}
    =
    \frac{\partial^2 f}{\partial \theta^2}
    +
    \frac{r(1+r^2)}{1-r^2} \frac{\partial f}{\partial r},\\
    (\Hess(f))_{33}
    =
    \frac{\partial^2 f}{\partial t^2},\\
     (\Hess(f))_{12}
    =
    \frac{\partial^2 f}{\partial r \partial \theta}
    -
    \Gamma_{12}^{2} \frac{\partial f}{\partial \theta}
    =
    \frac{\partial^2 f}{\partial r \partial \theta}
    -
    \frac{1+r^2}{r(1-r^2)} \frac{\partial f}{\partial \theta},\\
    (\Hess(f))_{13}
    =
    \frac{\partial^2 f}{\partial r \partial t},\\
    (\Hess(f))_{23}
    =
    \frac{\partial^2 f}{\partial \theta \partial t}.
\end{cases}$$

For a vector field
\begin{equation}\label{X}
    X(r,\theta,t)
=X_1(r,\theta,t)\frac{\partial}{\partial r}+X_2(r,\theta,t)\frac{\partial}{\partial\theta}+X_3(r,\theta,t)\frac{\partial}{\partial t}.
\end{equation} 
Using the local formula \begin{equation}
(\Lie_Xg)_{ij}=X^k\partial_kg_{ij}+g_{kj}\partial_iX^k+g_{ik}\partial_jX^k,
\label{eq:Lie}
\end{equation} the components of the Lie derivative are:

$$\begin{cases}
    (\Lie_Xg)_{11}
=
\frac{16r}{(1-r^2)^3}X_1
+
\frac{8}{(1-r^2)^2}
\frac{\partial X_1}{\partial r},\\
(\Lie_Xg)_{22}
=
\frac{8r(1+r^2)}{(1-r^2)^3}X_1
+
\frac{8r^2}{(1-r^2)^2}
\frac{\partial X_2}{\partial\theta},\\
(\Lie_Xg)_{33}
=
2\frac{\partial X_3}{\partial t},\\
(\Lie_Xg)_{12}
=
\frac{4r^2}{(1-r^2)^2}
\frac{\partial X_2}{\partial r}
+
\frac{4}{(1-r^2)^2}
\frac{\partial X_1}{\partial\theta},\\
(\Lie_Xg)_{13}
=
\frac{\partial X_3}{\partial r}
+
\frac{4}{(1-r^2)^2}
\frac{\partial X_1}{\partial t},\\
(\Lie_Xg)_{23}
=
\frac{\partial X_3}{\partial\theta}
+
\frac{4r^2}{(1-r^2)^2}
\frac{\partial X_2}{\partial t}.
\end{cases}$$
\begin{definition}
  Let $X$ be a smooth vector field on an $n$-dimensional Riemannian manifold $(M,g)$. Its dual form is given by the formula \begin{equation}\label{duale}
X^\flat(x_1,..,x_n)=\sum_{i=1}^n\sum_{j=1}^ng_{ij}X_idx_j.
   \end{equation}
\end{definition}
\begin{remark}
When $h = 1$ in equation \eqref{SRB} (respectively \eqref{SRBG}), the term used is simply Ricci-Bourguignon soliton (respectively gradient Ricci-Bourguignon soliton). Furthermore, if we simultaneously have $\rho = 0$, we obtain the classical notion of Ricci soliton (respectively gradient Ricci soliton). Moreover, an $h$-Ricci-Bourguignon soliton—just like a Ricci-Bourguignon or Ricci soliton is classified as:
\begin{itemize}[label=$\bullet$]
    \item \textbf{steady} if $\lambda = 0$;
    \item \textbf{shrinking} if $\lambda > 0$;
    \item \textbf{expanding} if $\lambda < 0$.
\end{itemize}
\end{remark}
\section{Main Results}
In this section, we present the results and their proofs under the assumption that $r \in (0, 1)$.
\begin{theorem}\label{T1}
Let $f$ and $h$ be two functions of at least class $\mathcal{C}^2(\mathbb{D}^2\times\R)$, with $h$ non-zero, and let $\lambda$ and $\rho$ be real constants. The quintuplet $(\mathbb{D}^2\times \R,g,\nabla f,\lambda,\rho)$ is a gradient $h$-Ricci-Bourguignon soliton if and only if the function $h$ depends at most on the variable $t$ and $\lambda=2\rho-1$. Furthermore, when setting $h(r,\theta,t):=\phi(t)$, where $\phi$ is a non-zero function of at least class $\mathcal{C}^2$ of a real variable, then \begin{equation}
    f(r,\theta,t)=-\int_{t_0}^t\left(\int_{\tau_0}^\tau\frac{\mathrm{d}s}{\phi(s)}\right)d\tau+B_0t+B_1
\end{equation} where $B_0$ and $B_1$ are real constants.
\end{theorem}
\begin{proof}
Suppose that $(\mathbb{D}^2\times \R,g,\nabla f,\lambda,\rho)$ is a gradient Ricci-Bourguignon $h$-soliton. Then equation \eqref{SRBG} is equivalent to the system
    \begin{equation}
\begin{cases}
\displaystyle
\frac{\partial^2 f(r,\theta,t)}{\partial r^2}
-
\frac{2r}{1-r^2} \frac{\partial f(r,\theta,t)}{\partial r}
-\frac{4}{h(r,\theta,t)(1-r^2)^2}
= \frac{4(\lambda-2\rho)}{h(r,\theta,t)(1-r^2)^2},
\\
\displaystyle
\frac{\partial^2 f(r,\theta,t)}{\partial \theta^2}
+
\frac{r(1+r^2)}{1-r^2} \frac{\partial f(r,\theta,t)}{\partial r}
-\frac{4r^2}{h(r,\theta,t)(1-r^2)^2}
=\frac{4r^2(\lambda-2\rho)}{h(r,\theta,t)(1-r^2)^2},
\\
\displaystyle
\frac{\partial^2 f(r,\theta,t)}{\partial t^2}
=
\frac{\lambda-2\rho}{h(r,\theta,t)},
\\
\displaystyle
\frac{2\partial^2 f(r,\theta,t)}{\partial r \partial\theta} 
-
\frac{1+r^2}{r(1-r^2)} \frac{\partial f(r,\theta,t)}{\partial \theta}
=
0,
\\
\displaystyle
\frac{\partial^2 f(r,\theta,t)}{\partial r \partial t}
=
0,
\qquad
\frac{\partial^2 f(r,\theta,t)}{\partial \theta \partial t}
=
0.
\end{cases}
\label{e1}
\end{equation}
The system \eqref{e1} is equivalent to: $$\begin{cases}
\frac{\partial^2 f(r,\theta,t)}{\partial r^2}
-\frac{2r}{1-r^2} \frac{\partial f(r,\theta,t)}{\partial r}= \frac{4(\lambda-2\rho+1)}{h(r,\theta,t)(1-r^2)^2}\quad (l_1),
\\
\frac{\partial^2 f(r,\theta,t)}{\partial \theta^2}+\frac{r(1+r^2)}{1-r^2} \frac{\partial f(r,\theta,t)}{\partial r}
=\frac{4r^2(\lambda-2\rho+1)}{h(r,\theta,t)(1-r^2)^2}\quad (l_2),
\\
\frac{\partial^2 f(r,\theta,t)}{\partial t^2}
=\frac{\lambda-2\rho}{h(r,\theta,t)}\quad (l_3),
\\
\frac{2\partial^2 f(r,\theta,t)}{\partial r \partial\theta} 
-\frac{1+r^2}{r(1-r^2)} \frac{\partial f(r,\theta,t)}{\partial \theta}=0\quad (l_4),
\\
\frac{\partial^2 f(r,\theta,t)}{\partial r \partial t}=0\quad (l_5),\\

\frac{\partial^2 f(r,\theta,t)}{\partial \theta \partial t}=0\quad (l_6).
\end{cases}$$
Line $(l_6)$ shows that $f(r,\theta,t)=f_1(r,\theta)+f_2(r,t)$. Applying this expression for $f$ to line $(l_5)$ yields $$f(r,\theta,t)=f_1(r,\theta)+f_3(r)+f_4(t).$$Line $(l_4)$ gives us: $$\frac{\partial^2f_1(r,\theta)}{\partial_r\partial_\theta}-\frac{1+r^2}{r(1-r^2)}\frac{\partial f_1(r,\theta)}{\partial_\theta}=0\iff \frac{\partial}{\partial\theta}\left(\frac{\partial f_1(r,\theta)}{\partial_r}-\frac{1+r^2}{r(1-r^2)}f_1(r,\theta)\right)=0$$ This equation shows that there exists at most one function $a$ depending on the variable $r$ such that $$\frac{\partial f_1(r,\theta)}{\partial r}-\frac{1+r^2}{r(1-r^2)}f_1(r,\theta)=a(r)$$
The homogeneous equation is:
$$\frac{\partial f_{1ho}}{\partial r} - \frac{1+r^2}{r(1-r^2)} f_{1ho} =0\quad\text{so } :\quad \frac{df_{1ho}}{f_{1ho}} = \frac{1+r^2}{r(1-r^2)} dr$$

By decomposing the rational fraction, we find: $\frac{1+r^2}{r(1-r^2)} = \frac{1}{r} + \frac{2r}{1-r^2}$

By integrating in relation to $r$ :
$$\int \left( \frac{1}{r} + \frac{2r}{1-r^2} \right) dr = \ln|r| - \ln|1-r^2| = \ln\left| \frac{r}{1-r^2} \right|$$

Integrating the homogeneous equation yields:
$$\ln|f_{1ho}| = \ln\left| \frac{r}{1-r^2} \right| + b(\theta)$$

Applying the exponential function, the homogeneous solution is written as:
$$f_{1ho}(r,\theta) = b(\theta) \frac{r}{1-r^2}$$
where $b(\theta)$ is an arbitrary function of $\theta$. Using the method of variation of constants, we seek a solution of the form:
$$f_{1\text{part}}(r,\theta) = K(r,\theta) \frac{r}{1-r^2}$$

Substituting this into the full equation yields:
$$\frac{\partial K}{\partial r} \frac{r}{1-r^2} = a(r)\iff\frac{\partial K}{\partial r} = a(r) \frac{1-r^2}{r}.$$

Integrating with respect to $r$:
$$K(r,\theta) = \int a(r) \frac{1-r^2}{r} dr.$$
We find: 
$$f_1(r,\theta) = \left( b(\theta) + \int a(r) \frac{1-r^2}{r} dr \right) \frac{r}{1-r^2}.$$
The function $f$ thus becomes: $$f(r,\theta,t)=\left( b(\theta) + \int a(r) \frac{1-r^2}{r} dr \right) \frac{r}{1-r^2}+f_3(r)+f_4(t).$$ Differentiating the equation in line $(l_3)$ with respect to $r$ and then with respect to $\theta$ shows that $\begin{cases}
    \partial_r\left(\frac{1}{h(r,\theta,t)}\right)=0,\\
    \partial_\theta\left(\frac{1}{h(r,\theta,t)}\right)=0
\end{cases}$, implying that the function $h$ depends at most on the variable $t$.

Differentiating lines $(l_1)$ and $(l_2)$ with respect to $t$ simultaneously shows that $\lambda=2\rho-1$. Line $(l_2)$ becomes: $$\frac{\partial^2 f(r,\theta,t)}{\partial \theta^2}+\frac{r(1+r^2)}{1-r^2} \frac{\partial f(r,\theta,t)}{\partial r}
=0$$  with $$f(r,\theta,t)=\left( b(\theta) + \int a(r) \frac{1-r^2}{r} dr \right) \frac{r}{1-r^2}+f_3(r)+f_4(t).$$
Let $A(r)=\int a(r)\frac{1-r^2}{r}\,dr.\quad \text{Then } A'(r)=a(r)\frac{1-r^2}{r},$
so  
\begin{equation}
f(r,\theta,t)=\bigl(b(\theta)+A(r)\bigr)\frac{r}{1-r^2}+f_3(r)+f_4(t).
\label{eq:f-A}
\end{equation}

Since the functions $A$, $f_3$, and $f_4$ are independent of
$\theta$, we obtain
$$\frac{\partial f}{\partial\theta}
=
b'(\theta)\frac{r}{1-r^2},\quad \text{and consequently } \frac{\partial^2 f}{\partial\theta^2}
=b''(\theta)\frac{r}{1-r^2}.$$
We also have, on the other side,
\begin{align*}
\frac{\partial f}{\partial r}
&=
A'(r)\frac{r}{1-r^2}
+
\bigl(b(\theta)+A(r)\bigr)
\frac{d}{dr}
\left(\frac{r}{1-r^2}\right)
+
f_3'(r).
\end{align*}

while, $\left(\frac{r}{1-r^2}\right)'
=
\frac{1+r^2}{(1-r^2)^2}$ and like $A'(r)\frac{r}{1-r^2}
=
a(r).$
therefore, $\frac{\partial f}{\partial r}=a(r)+\bigl(b(\theta)+A(r)\bigr)\frac{1+r^2}{(1-r^2)^2}+f_3'(r).$

In the end, we obtain $$0=b''(\theta)\frac{r}{1-r^2}+\frac{r(1+r^2)}{1-r^2}\left[a(r)+
\bigl(b(\theta)+A(r)\bigr)\frac{1+r^2}{(1-r^2)^2}+f_3'(r)\right].$$

We find
\begin{equation}
b''(\theta)+\frac{(1+r^2)^2}{(1-r^2)^2}b(\theta)+\frac{(1+r^2)^2}{(1-r^2)^2}A(r)+(1+r^2)\bigl(a(r)+f_3'(r)\bigr)
=0.
\label{eq}
\end{equation}
In \eqref{eq}, the only term that depends solely on $\theta$ is $b''(\theta)$, whereas $b(\theta)$ is multiplied by the function $\frac{(1+r^2)^2}{(1-r^2)^2}$, which depends non-trivially on $r$.

For identity \eqref{eq} to hold for all $(r,\theta)$, the function $b$ must therefore be constant. Thus, there exists a constant $b_0 \in \mathbb{R}$ such that $b(\theta) = b_0 \in \R$.

Equation \eqref{eq} then becomes
\begin{equation}\label{eqq}
a(r)+f_3'(r)+\bigl(b_0+A(r)\bigr)\frac{1+r^2}{(1-r^2)^2}=0.
\end{equation}

Observing that
\begin{align*}
\left[
\bigl(b_0+A(r)\bigr)
\frac{r}{1-r^2}
\right]'
&=A'(r)\frac{r}{1-r^2}+\bigl(b_0+A(r)\bigr)
\frac{1+r^2}{(1-r^2)^2}
\\
&=a(r)+\bigl(b_0+A(r)\bigr)
\frac{1+r^2}{(1-r^2)^2}.
\end{align*}

Consequently, \eqref{eqq} can be rewritten in the form
$$\left[\bigl(b_0+A(r)\bigr)\frac{r}{1-r^2}
\right]'+f_3'(r)=0.$$
So,
 $$\left[\bigl(b_0+A(r)\bigr)\frac{r}{1-r^2}+f_3(r)\right]'=0.$$
There therefore exists a constant $C \in \R$ such that $\bigl(b_0+A(r)\bigr)\frac{r}{1-r^2}+f_3(r)=C.$

It follows that $f_3(r)=C-\bigl(b_0+A(r)\bigr)
\frac{r}{1-r^2}.$

Returning to the definition of $A$, we obtain explicitly
$$f_3(r)=C-\left(b_0+\int a(r)\frac{1-r^2}{r}\,dr\right)\frac{r}{1-r^2}.$$

This yields
$$f(r,\theta,t)=\left(b_0+\int a(r)\frac{1-r^2}{r}\,dr\right)\frac{r}{1-r^2}+C-\left(b_0+\int a(r)\frac{1-r^2}{r}\,dr\right)\frac{r}{1-r^2}+f_4(t).$$

The terms depending on $r$ cancel each other out exactly. We are thus left with $f(r,\theta,t)=C+f_4(t).$ By setting $h(r,\theta,t)=\phi(t)$, line $(l_3)$ becomes $$f(r,\theta,t)=C+(\lambda-   2\rho)\int_{t_0}^t\left(\int_{\tau_0}^\tau\frac{\mathrm{d}s}{\phi(s)}\right)d\tau+B_0t+C_1.$$
\end{proof}
\begin{corollary}
  Under the hypotheses of Theorem \ref{T1}, when $\rho=0$, the quadruple $(\mathbb{D}^2\times \R,g,\nabla f,\lambda)$ is an expanding gradient Ricci $h$-soliton. Furthermore, when $h:=1$, then \begin{equation}
        f(r,\theta,t)=\frac{-1}{2}t^2+x_0t+y_0
    \end{equation} where $x_0$ and $y_0$ are real constants.
\end{corollary}
\begin{theorem}\label{T2}
Consider the vector field given in \eqref{X}, such that its components are at least of class $\mathcal{C}^1$. The quintuplet $(\mathbb{D}^2\times \R,g,X,\lambda,\rho)$ is a Ricci-Bourguignon soliton if and only if \begin{equation}\label{champ}
    \begin{cases}
    \lambda=2\rho-1\\
    X(r,\theta,t)=\left(c\theta+c_0\right)(1-r^2)\partial_r+\left(\frac{c}{r}+c_1\right)\partial_\theta+\left(-t+b_0\right)\partial_t
\end{cases}
\end{equation} where $b_0, c, c_0, c_1$ are real constants.    
\end{theorem}

\begin{proof}
   Suppose that $(\mathbb{D}^2\times \R,g,X,\lambda,\rho)$ is a Ricci-Bourguignon soliton. Then the equation \eqref{SRB} is equivalent to the system $$\begin{cases}
\displaystyle
\frac{\partial X_1}{\partial r}
+\frac{2r}{1-r^2}X_1
=\lambda-2\rho+1\quad(\ell_1),
\\\\
\displaystyle
\frac{\partial X_2}{\partial\theta}
+\frac{1+r^2}{r(1-r^2)}X_1
=\lambda-2\rho+1\quad(\ell_2),
\\\\
\displaystyle
\frac{\partial X_3}{\partial t}
=\lambda-2\rho\quad(\ell_3),\\\\
\displaystyle
\frac{\partial X_1}{\partial\theta}
+r^2\frac{\partial X_2}{\partial r}=0\quad (\ell_4),\\\\
\displaystyle
\frac{\partial X_3}{\partial r}
+\frac{4}{(1-r^2)^2}
\frac{\partial X_1}{\partial t}=0\quad (\ell_5),\\\\
\displaystyle
\frac{\partial X_3}{\partial\theta}
+\frac{4r^2}{(1-r^2)^2}
\frac{\partial X_2}{\partial t}=0\quad (\ell_6).
\end{cases}$$
The homogeneous solution to the equation for the line $(\ell_1)$ is $X_1^{ho}(r,\theta,t)=w(\theta,t)(1-r^2)$.

Let us set $X_1^{\text{part}}(r,\theta,t)=K(r,\theta,t)(1-r^2)$ as a particular solution. Then $$\partial_rK=\frac{\lambda-2\rho+1}{1-r^2}\iff \partial_rK=\frac{\lambda-2\rho+1}{2}\left(\frac{1}{1+r}+\frac{1}{1-r}\right)$$ which yields $$X_1(r,\theta,t)=\left(w(\theta,t)+\frac{\lambda-2\rho+1}{2}\ln\left(\frac{1+r}{1-r}\right)\right)(1-r^2).$$ 
Line $(\ell_4)$ implies that $$X_2(r,\theta,t)=\frac{\partial_\theta w(\theta,t)}{r}+w_1(\theta,t).$$
Line $(\ell_3)$ yields $$X_3(r,\theta,t)=(\lambda-2\rho)t+w_2(r,\theta).$$
Line $(\ell_5)$ gives us $$\partial_rw_2(r,\theta)+\frac{4\partial_tw(\theta,t)}{1-r^2}=0;$$ this equation holds only when $w(\theta,t)=w_3(\theta)t+w_4(\theta)$ so $$w_2(r,\theta)=-2w_3(\theta)\ln\left(\frac{1+r}{1-r}\right)+w_5(\theta).$$
Which once again results in  $$\begin{cases}
X_1(r,\theta,t)=\left(w_3(\theta)t+w_4(\theta)+\frac{\lambda-2\rho+1}{2}\ln\left(\frac{1+r}{1-r}\right)\right)(1-r^2),\\
X_2(r,\theta,t)=\frac{tw'_3(\theta)+w'_4(\theta)}{r}+w_1(\theta,t),\\
X_3(r,\theta,t)=(\lambda-2\rho)t-2w_3(\theta)\ln\left(\frac{1+r}{1-r}\right)+w_5(\theta).
\end{cases}$$
The line $(\ell_6)$ gives $$-2w'_3(\theta)\ln\left(\frac{1+r}{1-r}\right)+w'_5(\theta)+\frac{4r}{(1-r^2)^2}w'_3(\theta)+\frac{4r^2}{(1-r^2)^2}\partial_tw_1(\theta,t)=0$$ this equation is true only if $w_1(\theta,t)=w_6(\theta)$, $w_3(\theta)=a_0\in\R$ and $w_5(\theta)=b_0\in\R$
we have: $$\begin{cases}
X_1(r,\theta,t)=\left(a_0t+w_4(\theta)+\frac{\lambda-2\rho+1}{2}\ln\left(\frac{1+r}{1-r}\right)\right)(1-r^2),\\
X_2(r,\theta,t)=\frac{w'_4(\theta)}{r}+w_6(\theta),\\
X_3(r,\theta,t)=(\lambda-2\rho)t-2a_0\ln\left(\frac{1+r}{1-r}\right)+b_0.
\end{cases}.$$
The line $(\ell_2)$ becomes $$\frac{\partial X_2}{\partial\theta}
+\frac{1+r^2}{r(1-r^2)}X_1
=\lambda-2\rho+1$$ $$\frac{w''_4(\theta)}{r}+w'_6(\theta)+\frac{1+r^2}{r}\left(a_0t+w_4(\theta)+\frac{\lambda-2\rho+1}{2}\ln\left(\frac{1+r}{1-r}\right)\right)=\lambda-2\rho+1.$$  This equation remains true only when $$\begin{cases}
    a_0=0,\\
    w_4(\theta)=c\theta+c_0,\\
    w_6(\theta)=c_1\in\R,\\
    \lambda-2\rho+1=0.
\end{cases}$$
Finaly $$\begin{cases}
X_1(r,\theta,t)=\left(c\theta+c_0\right)(1-r^2),\\
X_2(r,\theta,t)=\frac{c}{r}+c_1,\\
X_3(r,\theta,t)=(\lambda-2\rho)t+b_0.
\end{cases}.$$ 
\end{proof}
\begin{corollary}
    Under the hypotheses of Theorem \ref{T2}, when $\rho=0$, the quadruple $(\mathbb{D}^2\times \R,g,X,\lambda)$ is an expanding Ricci soliton.
\end{corollary}
\begin{theorem}
 The dual form of the vector field given by \eqref{champ} in Theorem \ref{T2} is a contact structure if and only if the function $$(r,\theta,t)\mapsto (-t+b_0)\left( \frac{4c r^6 - 8c r^4 - 16c_1 r^3 - 16c r^2 + 4c}{r^2(1-r^2)^3} \right)\quad \text{is non-zero} .$$  Furthermore, its Reeb $\xi$ field is equal to $\xi(r,\theta,t) = \frac{1}{-t + b_0} \frac{\partial}{\partial t}.$
\end{theorem}
\begin{proof}
According to formula \eqref{duale}, we have: $$ X^\flat(r,\theta,t)=\frac{4\left(c\theta+c_0\right)}{1-r^2}\mathrm{d}r+\frac{4(c_1r+c)}{r(1-r^2)^2}\mathrm{d}\theta+\left(-t+b_0\right)\mathrm{d}t.$$
 Let's set  : $A(r,\theta) = \frac{4(c\theta+c_0)}{1-r^2}, \quad B(r) = \frac{4(c_1r+c)}{r(1-r^2)^2}, \quad C(t) = -t+b_0.$

So, $X^\flat=A\mathrm{d}r+B\mathrm{d}\theta + C\mathrm{d}t$.

By the definition of the exterior derivative:
$$\mathrm{d}X^\flat = \mathrm{d}A \wedge \mathrm{d}r + \mathrm{d}B \wedge \mathrm{d}\theta + \mathrm{d}C \wedge \mathrm{d}t$$
So we have: $$\mathrm{d}A=\frac{\partial A}{\partial r}\mathrm{d}r + \frac{\partial A}{\partial \theta}\mathrm{d}\theta+\frac{\partial A}{\partial t}\mathrm{d}t = \frac{8r(c\theta+c_0)}{(1-r^2)^2}\mathrm{d}r + \frac{4c}{1-r^2}\mathrm{d}\theta+0\mathrm{d}t$$
  By applying the outer product with $\mathrm{d}r$ :
    $$\mathrm{d}A \wedge \mathrm{d}r = \frac{4c}{1-r^2}\mathrm{d}\theta \wedge \mathrm{d}r = -\frac{4c}{1-r^2}\mathrm{d}r \wedge \mathrm{d}\theta.$$
Since $B$ depends only on $r$, $\mathrm{d}B = \frac{\partial B}{\partial r}\mathrm{d}r$. By the quotient rule:
    $$\frac{\partial B}{\partial r} = \frac{16c_1 r^3 + 20c r^2 - 4c}{r^2(1-r^2)^3}$$
  Hence:
    $$\mathrm{d}B \wedge \mathrm{d}\theta = \frac{16c_1 r^3 + 20c r^2 - 4c}{r^2(1-r^2)^3}\mathrm{d}r \wedge \mathrm{d}\theta.$$
     Since $C$ depends only on $t$, $\mathrm{d}C = -\mathrm{d}t$, which yields $\mathrm{d}C \wedge \mathrm{d}t=-\mathrm{d}t\wedge \mathrm{d}t = 0$.

By gathering the terms for $\mathrm{d}X^\flat$ :
$$\mathrm{d}X^\flat = \left( -\frac{4c}{1-r^2} + \frac{16c_1 r^3 + 20c r^2 - 4c}{r^2(1-r^2)^3} \right)\mathrm{d}r \wedge \mathrm{d}\theta.$$

Which gives:
$$\mathrm{d}X^\flat = \frac{-4c r^6 + 8c r^4 + 16c_1 r^3 + 16c r^2 - 4c}{r^2(1-r^2)^3}\mathrm{d}r \wedge \mathrm{d}\theta.$$

We now have:
$$X^\flat \wedge \mathrm{d}X^\flat = \left( A\,\mathrm{d}r + B\,\mathrm{d}\theta + C\,\mathrm{d}t \right) \wedge \left( \frac{-4c r^6 + 8c r^4 + 16c_1 r^3 + 16c r^2 - 4c}{r^2(1-r^2)^3}\mathrm{d}r \wedge \mathrm{d}\theta \right)$$
Upon expansion, the terms containing $\mathrm{d}r \wedge \mathrm{d}r$ or $\mathrm{d}\theta \wedge \mathrm{d}r \wedge \mathrm{d}\theta$ vanish due to antisymmetry. Only the term in $\mathrm{d}t$ remains:
$$X^\flat \wedge \mathrm{d}X^\flat = C(t) \cdot \left( \frac{-4c r^6 + 8c r^4 + 16c_1 r^3 + 16c r^2 - 4c}{r^2(1-r^2)^3} \right) \mathrm{d}t \wedge \mathrm{d}r \wedge \mathrm{d}\theta$$
By reordering the differentials into the standard order $(\mathrm{d}r \wedge \mathrm{d}\theta \wedge \mathrm{d}t = -\mathrm{d}t \wedge \mathrm{d}r \wedge \mathrm{d}\theta)$ :
$$X^\flat \wedge \mathrm{d}X^\flat = \left[ (-t+b_0)\left( \frac{4c r^6 - 8c r^4 - 16c_1 r^3 - 16c r^2 + 4c}{r^2(1-r^2)^3} \right)\right] \mathrm{d}r \wedge \mathrm{d}\theta \wedge \mathrm{d}t$$

For $X^\flat \wedge \mathrm{d}X^\flat \neq 0$, it is necessary and sufficient that the function $$(r,\theta,t)\mapsto (-t+b_0)\left( \frac{4c r^6 - 8c r^4 - 16c_1 r^3 - 16c r^2 + 4c}{r^2(1-r^2)^3} \right)$$ be non-zero. To determine the Reeb field $\xi$ associated with the dual form $X^\flat$, we use the two fundamental conditions characterizing $\xi$: \begin{enumerate}
    \item $X^\flat(\xi) = 1,$
\item $\iota_\xi \mathrm{d}X^\flat =0.$
\end{enumerate}
Let $\xi$ be a generic vector field on the manifold $\mathbb{D}^2\times \R$, expressed in the coordinate basis:
$$\xi = \xi^r \frac{\partial}{\partial r} + \xi^\theta \frac{\partial}{\partial \theta} + \xi^t \frac{\partial}{\partial t}.$$
Let $D(r) = \frac{-4c r^6 + 8c r^4 + 16c_1 r^3 + 16c r^2 - 4c}{r^2(1-r^2)^3}$. Let us compute the interior product (contraction) of $\xi$ with $\mathrm{d}X^\flat$:
$$\iota_\xi \mathrm{d}X^\flat = \iota_\xi \left( D(r) \, \mathrm{d}r \wedge \mathrm{d}\theta \right) = D(r) \left( \iota_\xi(\mathrm{d}r) \mathrm{d}\theta - \iota_\xi(\mathrm{d}\theta) \mathrm{d}r \right)$$
$$\iota_\xi \mathrm{d}\alpha = D(r) \left( \xi^r \mathrm{d}\theta - \xi^\theta \mathrm{d}r \right)$$

For this expression to vanish identically (since $D(r) \neq 0$ on the valid domain), the spatial components of $\xi$ must vanish:
$$\xi^r = 0 \quad \text{and} \quad \xi^\theta = 0.$$
Let us evaluate the action of the 1-form $X^\flat$ on the field $\xi = \xi^t \frac{\partial}{\partial t}$:
$$X^\flat(\xi) = A(r,\theta)\xi^r + B(r)\xi^\theta + C(t)\xi^t = 1$$

Substituting $\xi^r = 0$ and $\xi^\theta = 0$, and knowing that $C(t) = -t + b_0$:
$$(-t + b_0) \xi^t = 1$$

We deduce the temporal component:
$$\xi^t = \frac{1}{-t + b_0}.$$

The Reeb field $\xi$ associated with the 1-form $X^\flat$ is therefore:
$$\xi(r,\theta,t) = \frac{1}{-t + b_0} \frac{\partial}{\partial t}.$$
\end{proof}


\begin{thebibliography}{99}


\bibitem{Besse}
A.~L.~Besse,
\newblock \emph{Einstein Manifolds},
\newblock Springer-Verlag, Berlin, 1987.


\bibitem{BoussoNdiayeJDSGT2025}
A.~Bousso and A.~Ndiaye,
\newblock $\eta$-Ricci-Bourguignon soliton on the hyperbolic spaces,
\newblock \emph{Journal of Dynamical Systems and Geometric Theories}
(2026).
\newblock DOI:
\doi{10.47974/JDSGT-2025-09003}.

\bibitem{BoussoNdiaye2026Hn}
A.~Bousso and A.~Ndiaye,
\newblock Ricci solitons on the Poincaré upper half plane,
\newblock \emph{Honam Mathematical Journal}
\textbf{48} (2026), no.~2, 329--348.
\newblock DOI:
\doi{10.5831/HMJ.2026.48.2.329}.

\bibitem{BoussoNdiayeH2R2025}
A.~Bousso and A.~Ndiaye,
\newblock $h$-Ricci-Bourguignon solitons on the
$\mathbb{H}^2\times\mathbb{R}$ Lie group,
\newblock \emph{Global Journal of Advanced Research on
Classical and Modern Geometries}
\textbf{14} (2025), no.~2, 200--207.

\bibitem{BoussoNdiayeSol3}
A.~Bousso and A.~Ndiaye,
\newblock Ricci--Yamabe solitons on the Lie group
$\mathrm{Sol}_3$,
\newblock \emph{Journal of Universal Mathematics}
\textbf{9} (2026), no.~1, 36--45.
\newblock DOI:
\doi{10.33773/jum.1906429}.


\bibitem{Chow1992}
B.~Chow,
\newblock The Yamabe flow on locally conformally flat manifolds with positive Ricci curvature,
\newblock \emph{Comm. Pure Appl. Math.}
\textbf{45} (1992), no.~8, 1003--1014.

\bibitem{DiopBoussoNdiayeMandal2026}
M.~N.~Diop, A.~Bousso, A.~Ndiaye and A.~Mandal,
\newblock $(h,\eta)$-Ricci-Bourguignon soliton on the
Poincaré Disk $D^2$,
\newblock \emph{Konuralp Journal of Mathematics}
\textbf{14} (2026), no.~1.

\bibitem{doCarmo}
M.~P.~do Carmo,
\newblock \emph{Riemannian Geometry},
\newblock Birkhäuser, Boston, 1992.


\bibitem{Hamilton}
R.~S.~Hamilton,
\newblock The Ricci flow on surfaces,
\newblock \emph{Mathematics and General Relativity},
\newblock Contemporary Mathematics \textbf{71} (1988), 237--262.
\bibitem{Lee}
J.~M.~Lee,
\newblock \emph{Riemannian Manifolds: An Introduction to Curvature},
\newblock Graduate Texts in Mathematics, Vol.~176,
\newblock Springer, New York, 1997.
\bibitem{ONeill}
B.~O'Neill,
\newblock \emph{Semi-Riemannian Geometry with Applications to Relativity},
\newblock Academic Press, New York, 1983.
\bibitem{Thurston1997}
W.~P.~Thurston,
\newblock \emph{Three-Dimensional Geometry and Topology},
\newblock Vol.~1, Princeton University Press, Princeton, 1997.
\end{thebibliography}
\end{document}